\documentclass[11pt]{article}

\usepackage[T1]{fontenc}
\usepackage[utf8]{inputenc}

\usepackage{fourier} 
\usepackage{soul} 
\usepackage{enumerate} 

\usepackage{amsfonts,amssymb,amsmath,amsthm}
\usepackage{bbm}
\usepackage{pgf}

\usepackage{hyperref}
\usepackage{cancel}

\usepackage{xcolor}
\definecolor{darkblue}{rgb}{0,0,0.5}
\usepackage{transparent}
\usepackage{graphicx}
\newcommand{\executeiffilenewer}[3]{%
  \ifnum\pdfstrcmp{\pdffilemoddate{#1}}%
  {\pdffilemoddate{#2}}>0%
  {\immediate\write18{#3}}\fi%
}

\numberwithin{equation}{section}
\def\PP{\mathbb{P}}
\def\QQ{\mathbb{Q}}
\def\RR{\mathbb{R}}

\def\ZZ{\mathbb{Z}}

\def\EE{\mathbb{E}}
\def\11{\mathbbm{1}}

\def\E{\mathbb{E}}
\def\P{\mathbb{P}}

\def\Q{\mathbb{Q}}

\def\d{\partial}
\def\Z{\mathbb{Z}}

\newtheorem{thm}{Theorem}[section]

\newtheorem{cor}[thm]{Corollary}

\newtheorem{prop}[thm]{Proposition}

\theoremstyle{remark}
\newtheorem{rem}{Remark}

\newcommand{\vertiii}[1]{{\left\vert\kern-0.25ex\left\vert\kern-0.25ex\left\vert #1 
    \right\vert\kern-0.25ex\right\vert\kern-0.25ex\right\vert}}

    \def\restriction#1#2{\mathchoice
                  {\setbox1\hbox{${\displaystyle #1}_{\scriptstyle #2}$}
                  \restrictionaux{#1}{#2}}
                  {\setbox1\hbox{${\textstyle #1}_{\scriptstyle #2}$}
                  \restrictionaux{#1}{#2}}
                  {\setbox1\hbox{${\scriptstyle #1}_{\scriptscriptstyle #2}$}
                  \restrictionaux{#1}{#2}}
                  {\setbox1\hbox{${\scriptscriptstyle #1}_{\scriptscriptstyle #2}$}
                  \restrictionaux{#1}{#2}}}
    \def\restrictionaux#1#2{{#1\,\smash{\vrule height .8\ht1 depth .85\dp1}}_{\,#2}}

\begin{document}

\title{Quasi-limiting estimates for periodic absorbed Markov chains}

\author{Nicolas Champagnat$^{1}$, Denis Villemonais$^{1}$}

\footnotetext[1]{Universit\'e de Lorraine, CNRS, Inria, IECL, F-54000 Nancy, France \\
  E-mail: Nicolas.Champagnat@inria.fr, Denis.Villemonais@univ-lorraine.fr}

\maketitle

\begin{abstract}
  We consider periodic Markov chains with absorption. Applying to iterates of this periodic Markov chain criteria for the exponential
  convergence of conditional distributions of aperiodic absorbed Markov chains, we obtain exponential estimates for the periodic
  asymptotic behavior of the semigroup of the Markov chain. This implies in particular the exponential convergence in total variation
  of the conditional distribution of the Markov chain given non-absorption to a periodic sequence of limit measures and we
  characterize the cases where this sequence is constant, which corresponds to the cases where the conditional distributions converge
  to a quasi-stationary distribution. We also characterize the first two eignevalues of the semigroup and give a bound for the
  spectral gap between these eigenvalues and the next ones. Finally, we give ergodicity estimates in total variation for the Markov
  chain conditioned to never be absorbed, often called $Q$-process, and quasi-ergodicity estimates for the original Markov chain.
\end{abstract}




\section{Introduction}
\label{sec:intro}

Let $(X_t,t\in \ZZ_+)$ be a Markov chain in $E\cup\{\d\}$ where $E$ is a measurable space, $\d\not\in E$ and $\ZZ_+:=\{0,1,\ldots\}$. For all $x\in E\cup\{\d\}$, we denote as usual by $\PP_x$ the law of $X$ given $X_0=x$ and for any probability measure $\mu$ on $E\cup\{\d\}$, we define $\PP_\mu=\int_{E\cup\{\d\}}\PP_x\,\mu(dx)$. We also denote by $\E_x$ and $\E_\mu$ the associated expectations. We assume that $\d$ is absorbing, which means that $X_t=\d$ for all $t\geq \tau_\d$, $\PP_x$-almost surely, where 
\[ 
\tau_\d=\inf\left\{t\in I,\,X_t=\d\right\}.  
\]
We study the sub-Markovian transition semigroup of $X$ in $E$, $(P_n)_{n\in\Z_+}$, defined as
\begin{align*}
P_n f(x)=\E_x\left(f(X_n)\11_{n<\tau_\d}\right),\ \forall n\in\Z_+,
\end{align*}
for all bounded or nonnegative measurable function $f$ on $E$ and all $x\in E$. We also define as usual the left-action of $P_n$ on
measures as
\[
\mu P_n f=\EE_\mu\left(f(X_n)\11_{n<\tau_\d}\right)=\int_E P_nf(x)\,\mu(dx),
\]
for all positive measure $\mu$ on $E$ and all bounded measurable $f$.

Many references (see for
example~\cite{GongQianEtAl1988,ChampagnatVillemonais2016b,ChampagnatVillemonais2017b,ChampagnatVillemonais2019,FerreRoussetStoltz2018,HinrichKolbEtAl2018,BansayeCloezEtAl2019,GuillinNectouxEtAl2020,GuillinNectouxEtAl2021,LelievreRamilEtAl2021,BenaimChampagnatEtAl2021})
provide criteria allowing to characterize the asymptotic behavior of the semigroup $(P_n)_{n\in\ZZ_+}$ in the following form: there
exist a measurable function $V:E\rightarrow[1,+\infty)$, constants $\theta_0\in (0,1]$, $\alpha\in(0,1)$ and $C\in\mathbb{R}_+$, a
probability measure $\nu_\text{QS}$ on $E$ such that $\nu_\text{QS}(V)<+\infty$ and a measurable function $\eta:E\rightarrow\mathbb{R}_+$
non-identically zero such that $\eta/V$ is bounded and $\nu_\text{QS}(\eta)>0$, such that
\begin{equation}
  \label{eq:extension-eta}
  \left|\theta_0^{-n} P_n f(x)-\eta(x)\nu_\text{QS}(f)\right|\leq C\alpha^n V(x)
\end{equation}
for all measurable function $f:E\rightarrow\mathbb{R}$ such that $|f|\leq V$, all $n\in\mathbb{Z}_+$ and all $x\in E$.

All the previously cited references assume some property of aperiodicity for the process $X$, and~\eqref{eq:extension-eta} itself
implies a weak form of aperiodicity that can be formulated as follows: for all $x\in E$ such that $\eta(x)>0$ and all measurable
$A\subset E$ such that $\nu_\text{QS}(A)>0$, $P_n\mathbbm{1}_A(x)>0$ for all $n$ large enough. The purpose of this note is to explain
how~\eqref{eq:extension-eta} should be modified for periodic Markov chains, and to examine the implications of this modified property
for quasi-limiting estimates, spectral properties of the semigroup, ergodicity of the $Q$-process and quasi-ergodicity of the Markov
chain $X$. In the rest of this introduction, we shall recall all these properties when~\eqref{eq:extension-eta} holds true.


The property~\eqref{eq:extension-eta} implies that, for all $x\in E$ such that $\eta(x)>0$ and all measurable $A\subset E$,
\[
\left|\frac{P_n \mathbbm{1}_A(x)}{P_n\mathbbm{1}_E(x)}-\nu_\text{QS}(A)\right|\xrightarrow[n\rightarrow+\infty]{} 0,
\]
which means that $\nu_\text{QS}$ is a \emph{quasi-limiting distribution} for the process $X$ (see e.g.~\cite{MeleardVillemonais2012}), and
thus a \emph{quasi-stationary distribution} (QSD) of $X$ (see again ~\cite{MeleardVillemonais2012}), i.e.\ a probability measure
$\nu$ such that
\[
\PP_\nu(X_t\in\cdot\mid t<\tau_\partial)=\nu(\cdot),\quad\forall t\in \Z_+.
\]
In addition, the \emph{absorption rate} of the QSD $\nu_\text{QS}$ is $\theta_0$, i.e.\ $\mathbb{P}_{\nu_\text{QS}}(\tau_\d\geq n)=\theta_0^n$
(see again~\cite{MeleardVillemonais2012}), and $\nu_\text{QS}(\eta)=1$. It is also easy to see that the measure $\nu_\text{QS}$ is the unique
QSD satisfying $\nu_\text{QS}(\eta)>0$ and $\nu_\text{QS}(V)<+\infty$ and that the function $\eta$ is an eigenfunction of the semigroup, since
$P_n\eta=\theta_0^n\eta$ for all $n\geq 0$. After giving the setting of our results on periodic processes and basic properties of
their QSDs in Section~\ref{sec:QSD}, the quasi-limiting properties recalled above are extended to the periodic case in
Section~\ref{sec:QLD} below. Of course, exponential convergence of the conditional distributions does not hold in general because of
the periodicity of the process~\cite{FerrariKestenMartinez1996,DoornPollett2008,DoornPollett2009,Ocafrain2018}, yet we are able to
obtain periodic exponential estimates in total variation.

The property~\eqref{eq:extension-eta} also has consequences on the spectrum of the semigroup $(P_n)_{n\geq 0}$: given a nonzero
function $h:E\cup\{\d\}\rightarrow\mathbb{C}$ such that $h/V$ is bounded and $\lambda h(x)=\mathbb{E}_x(h(X_1))$ for all $x\in
E\cup\{\d\}$ for some $\lambda\in\mathbb{C}$,
\begin{itemize}
\item either $h(\d)\neq 0$, and then $\lambda=1$ and $P_1(h-h(\d))(x)=h(x)-h(\d)$ for all $x\in E$, so it follows
  from~\eqref{eq:extension-eta} that the function $h$ is constant,
\item or $h(\d)=0$ and
  \begin{itemize}
  \item either $\nu_\text{QS}(h)\neq 0$ and then it follows from~\eqref{eq:extension-eta} that $\lambda=\theta_0$ and $h=\nu_\text{QS}(h)\eta$,
  \item or $\nu_\text{QS}(h)=0$ and then it follows from~\eqref{eq:extension-eta} that $|\lambda|\leq\alpha\theta_0$.
  \end{itemize}
\end{itemize}
This last property quantifies the spectral gap of the operator $P_1$. These spectral properties are extended to the periodic case in
Section~\ref{sec:spectral} below.

In addition, assuming that the Markov chain $(X_n)_{n\geq 0}$ is an adapted stochastic process defined on a filtered probability
space $(\Omega,\{\mathcal{F}_m\}_{m\geq 0},\{\mathbb{P}_x\}_{x\in E\cup\{\d\}})$, it follows from~\cite{champagnat2020}
and~\cite{champagnat2017general} that, under~\eqref{eq:extension-eta}, for all $x\in E'$ where $E':=\{y\in E:\eta(x)>0\}$, the limit
\[
\QQ_x(A)=\lim_{n\rightarrow+\infty}\PP_x(A\mid n<\tau_\partial)
\]
for all $A\in\mathcal{F}_m$ with $m\in\mathbb{N}$, defines a probability measure on $\Omega$ under which the process $(\Omega,({\cal
  F}_m)_{m\geq 0},(X_n)_{n\geq 0},(\QQ_x)_{x\in E'})$ is an $E'$-valued homogeneous Markov chain called the $Q$-process. In addition,
this process is exponentially ergodic with unique invariant measure $\eta\, d\nu_\text{QS}$ in the following sense: for all measurable
$f:E'\rightarrow\mathbb{R}$ such that $|f|\leq V/\eta$, for all $x\in E'$,
\[
\left|\EE_{\QQ_x}(f(X_n))-\int_{E'}f(x)\eta(x)\nu_\text{QS}(dx) \right|\leq C'\alpha^n\frac{V(x)}{\eta(x)}
\]
for constants $C'\in\mathbb{R}_+$ and $\alpha\in(0,1)$ as in~\eqref{eq:extension-eta}. These properties of the $Q$-process are
extended to the periodic case in Section~\ref{sec:Q-proc} below.

To conclude,~\eqref{eq:extension-eta} also implies the following quasi-ergodicity property, as shown in~\cite{Ocafrain-PhD-2019} [ou autre
r\'ef\'erence de William ? Le papier sur (E') ?] (see also~\cite{BreyerRoberts1999,ChampagnatVillemonais2016}): for all
measurable function $f:E\rightarrow[-1,1]$ and all $x\in E$ and $n\geq 1$,
\[
\left|\EE_x\left[\left.\frac{1}{n}\sum_{k=0}^{n}f(X_k)\ \right|\ n<\tau_\d\right]-\int_{E}f(x)\eta(x)\nu_\text{QS}(dx)\right|\leq
\frac{C''}{n}
\]
for some constant $C''\in\mathbb{R}_+$ independent on $n$. This quasi-ergodic property is extended to the periodic case and improved in
Section~\ref{sec:QED} below.

\begin{rem} 
  \label{rem:unbounded-sg}
  All the results of this note can be easily extended to general unbounded semi-groups (i.e. not necessarily sub-Markov) following
  the same approach as in~\cite{ChampagnatVillemonais2019}. We restrict the presentation to sub-Markov semigroups to make simpler the
  probabilistic interpretation of our results.
\end{rem}

\section{General assumptions and first properties of quasi-sta\-tio\-na\-ry distributions for periodic Markov chains}
\label{sec:QSD}

In all the sequel, we shall make the standing assumption that $E$ is the disjoint union of measurable sets $A_0,\ldots,A_{t-1}$ such
that for all $x\in A_i$,
\begin{equation}
  \label{eq:periodic}
  P_1\11_{A_{j}}(x)=
  \begin{cases}
    0 & \text{if }j\neq i+1, \\
    P_1\11_E(x) & \text{if }j=i+1,
  \end{cases}  
\end{equation}
with the convention $i+1=0$ if $i=t-1$. Then $P$ is $t$ periodic and $Q_n=P_{nt}$ for all $n\geq 0$ defines a sub-Markov semi-group
on the set of bounded measurable function on $A_0$.

\begin{rem}
  The last assumption corresponds to a particular case of periodicity, which assumes a (weak) form of irreducibility. Other periodic
  situations may occur when the state space is not irreducible and different periods may exist in different irreducibility classes,
  or when the process becomes periodic after leaving a transient set. In such situations, the results of this not may be extended
  combining our arguments with e.g.\ those of~\cite{champagnat2022quasistationary}.
\end{rem}

We first observe that, under the above conditions, there is a one-to-one correspondance between QSDs for $Q$ and $P$.

\begin{prop}
  \label{prop:one-to-one-QSD}
  If $\nu_\text{QS}$ is a QSD for $P$ with absorption rate $\theta>0$, then $\nu:=\frac{\nu_\text{QS}(\cdot\cap A_0)}{\nu_\text{QS}(A_0)}$ is a QSD for $Q$ with
  absorption rate $\theta^t$.\\
  Conversely, if a probability measure $\nu$ on $A_0$ is a QSD for $Q$ with absorption rate $\theta$, then the probability measure on
  $E$ defined as 
  \begin{equation}
    \label{eq:QSD-periodic}
    \nu_\text{QS}
    =\frac{1}{\sum_{i=0}^{t-1}\theta^{-i/t}\,\nu
      P_i\11_E}\sum_{i=0}^{t-1}\theta^{-i/t}\,\nu P_i     
  \end{equation}
  is a QSD for $P$ with absorption rate $\theta^{1/t}$, and it is the only one such that $\nu=\frac{\nu_\text{QS}(\cdot\cap A_0)}{\nu_\text{QS}(A_0)}$.
\end{prop}

\begin{proof}
  Observing that, if $\nu_\text{QS}$ is a QSD for $P$, then $\nu_\text{QS}(A_0)>0$ (since otherwise $\nu_\text{QS}(A_i)$ would be zero for all $i$
  using the QSD property), the first statement is clear. For the second statement, we first notice that, if $\nu$ is a QSD for $Q$,
  then it is clear that $\nu_\text{QS}$ as defined in~\eqref{eq:QSD-periodic} is a QSD for $P$. So it only remains to check the uniqueness
  of this QSD. Let $\mu$ be any QSD for $P$ such that $\nu=\frac{\mu(\cdot\cap A_0)}{\mu(A_0)}$. It follows from the first part of
  Proposition~\ref{prop:one-to-one-QSD} that its absorption rate is $\theta^{1/t}$. It then follows from~\eqref{eq:periodic} that, for all
  $i\in\{0,\ldots,t-1\}$,
  \[ 
  \theta^{i/t} \mu(\cdot\cap A_i)=\mu P_i (\cdot\cap
  A_i)=\mu(\cdot\cap A_0) P_i=\mu(A_0)\,\nu P_i.
  \] 
  Therefore,
  \[
  \mu=\mu(A_0)\sum_{i=0}^{t-1}\theta^{-i/t}\,\nu P_i,
  \]
  which entails~\eqref{eq:QSD-periodic}.
\end{proof}

In the next result, we observe that the one-to-one correspondance between QSDs for $(P_n)_{n\in\ZZ_+}$ and $(Q_k)_{k\in\ZZ_+}$ does
not extend to a one-to-one correspondance to the QSDs for $(P_n)_{n\in\ZZ_+}$ and $(P_{kt})_{k\in\ZZ_+}$, which will explain the
periodic asymptotic behaviors observed in the results of the next section.

\begin{prop}
  \label{prop:non-uniqueness-QSD}
  If $\nu$ is a QSD for $(Q_k)_{k\in\ZZ_+}$, then all convex combinations of 
  \begin{equation}
    \label{eq:QSDs}
    \nu,\ \frac{\nu P_1}{\nu P_1\11_E},\ldots,\ \frac{\nu P_{t-1}}{\nu P_{t-1}\11_E}    
  \end{equation}
  are QSDs for $(P_{kt})_{k\in\ZZ_+}$. Among them, only a single one is a QSD for $(P_n)_{n\in\ZZ_+}$.
\end{prop}

\begin{proof}
  It is clear that all the probability measures in~\eqref{eq:QSDs} are QSDs for $(P_{kt})_{k\in\ZZ_+}$ with the same absorption rate,
  so all convex combinations of these measures are QSDs. The unique one among them which is a QSD for $(P_n)_{n\in\ZZ_+}$ is given
  by~\eqref{eq:QSD-periodic}, by Proposition~\ref{prop:one-to-one-QSD}.
\end{proof}

In the following sections, we shall assume in addition that $Q$ satisfies the property~\eqref{eq:extension-eta} on $A_0$, i.e.\ there
exist a measurable function $V:A_0\rightarrow[1,+\infty)$, constants $\theta_0\in(0,1]$, $C\in\RR_+$ and $\alpha\in(0.1)$, a nonzero
measurable function $\eta:A_0\rightarrow\mathbb{R}_+$ such that $\eta/V$ is bounded and a probability measure $\nu$ on $A_0$ such
that $\nu(V)<+\infty$ and $\nu(\eta)>0$, such that
\begin{equation}
  \label{eq:lemme-note}
  \left|\theta_0^{-kt} Q_k f(x)-\eta(x)\nu(f)\right|\leq C_Q\alpha^k V(x)
\end{equation}
for all mesurable function $f$ on $A_0$ such that $|f|\leq V$, all $k\geq 0$ and all $x\in A_0$. Note that the absorption rate
$\theta_0$ in~\eqref{eq:extension-eta} has been modified above as $\theta_0^t$, so that the QSD $\nu_\text{QS}$ defined
in~\eqref{eq:QSD-periodic} has absorption rate $\theta_0$ instead of $\theta_0^{1/t}$. Note also that the last inequality implies,
setting $k=1$, that, for all $x\in A_0$,
\begin{equation}
  \label{eq:borne-triviale}
  Q_1 V(x)\leq \left(\|\eta/V\|_\infty \nu(V)+C_Q\alpha\right) V(x).
\end{equation}

\section{Quasi-limiting behavior of periodic Markov chains under~\eqref{eq:lemme-note}}
\label{sec:QLD}

Our first results extends~\eqref{eq:extension-eta} to periodic Markov chains. Let us define
\[
\mathcal{B}_{V}=\left\{f:E\rightarrow\mathbb{R}\text{ measurable, s.t.\ }\forall i\in\{0,\ldots,t-1\},\
  |P_i(f\11_{A_i})|\leq V\text{ on }A_0\right\}.
\]
Note that this set is non-empty since, by~\eqref{eq:borne-triviale}, it contains all functions of the form $P_j g$ for any
$j\in\{0,\ldots,t-1\}$ and any function $g$ such that $|g|\leq\frac{V}{\|\eta/V\|_\infty \nu(V)+C\alpha}$ on $A_0$ and $g\equiv 0$ on
$E\setminus A_0$. Since $V\geq 1$ and $P$ is sub-Markov, it also contains all measurable functions bounded by $1$.

\begin{thm}
  \label{thm:eta-periodic}
  Assume that $P$ satisfies~\eqref{eq:periodic} and that $Q$ satisfies~\eqref{eq:lemme-note}. Then, there exists $C<+\infty$ such
  that, for all $f\in\mathcal{B}_{V}$, $n\geq 1$, $j\in\{0,\ldots,t-1\}$ and $x\in A_k$ for some $k\in\{0,\ldots,t-1\}$,
  \begin{equation}
    \label{eq:main-eta-periodic}
    \left|\theta_0^{-(nt+j)}P_{nt+j} f(x)-\theta_0^{-(t+j)}P_{t-k}\eta(x) \,\nu P_{k+j} f\right|\leq C'_Q\alpha^n P_{t-k} V(x),
  \end{equation}
  where we extended $\nu$ by 0 to $E\setminus A_0$, where the constant $\alpha\in(0,1)$ is the same as in~\eqref{eq:lemme-note}
  and where
  \[
    C'_Q=C_Q\theta_0^{-2t}\left(\|\eta/V\|_\infty \nu(V)+C_Q\alpha\right)^2.
  \]
\end{thm}

\begin{proof}
  Let $f\in\mathcal{B}_{V}$, $n\geq 1$ and $j\in\{0,\ldots,t-1\}$. We have
  \begin{multline*}
    \left|\theta_0^{-(nt+j)}P_{nt+j} f-\theta_0^{-(t+j)} \sum_{i=0}^{t-1} P_{t+j-i}\eta\,\nu P_i f\right| \\
    \begin{aligned}
      & \leq \theta_0^{-(t+j)}\sum_{i=0}^{t-1} \left|\theta_0^{-(n-1)t}P_{t+j-i} P_{(n-1)t} P_i (f\11_{A_i})- P_{t+j-i}\eta\,\nu P_i
        (f\11_{A_i})\right| \\
      & \leq \theta_0^{-(t+j)}\sum_{i=0}^{t-1} P_{t+j-i} \left|\theta_0^{-(n-1)t} Q_{n-1} P_i (f\11_{A_i})-\nu P_i
        (f\11_{A_i})\,\eta\right| \\
      & \leq C_Q \theta_0^{-(t+j)} \,\alpha^{n-1}\sum_{i=0}^{t-1} P_{t+j-i}V.
    \end{aligned}
  \end{multline*}
  Now, given $x\in A_k$ for some $k\in\{0,\ldots,t-1\}$, we observe that only a single term in the sum in the left-hand side of the
  last equation is nonzero, corresponding to the unique index $i\in\{0,1,\ldots,t-1\}$ such that $j-i=-k$, $j-i=t-k$ or $j-i=2t-k$. In the first
  case, this term is $P_{t-k}\eta(x)\,\nu P_{j+k} f$; in the second case, it is 
  \begin{align*}
    P_{2t-k}\eta(x)\,\nu P_{j+k-t} f & =P_{t-k}Q_1\eta(x)\,\nu P_{j+k-t} (f\11_{A_{j+k-t}}) \\ & =P_{t-k}\eta(x)\,\theta_0^{t}\nu P_{j+k-t}
    (f\11_{A_{j+k-t}}) \\ & =P_{t-k}\eta(x)\,\nu Q_1 P_{j+k-t} (f\11_{A_{j+k-t}})=P_{t-k}\eta(x)\,\nu P_{j+k} f;
  \end{align*}
  and the third case can be handled similarly.
  So~\eqref{eq:main-eta-periodic} is proved using~\eqref{eq:borne-triviale}.
\end{proof}

As in the aperiodic case, the last result implies geometric estimates in $V$-weighted total variation for the long time behavior of the
conditional distributions of the Markov chain given non-absorption.

\begin{thm}
  \label{thm:QLD-periodic}
  Under the assumptions of Theorem~\ref{thm:eta-periodic}, there exist constants $C<+\infty$ and $\bar{\alpha}\in(0,\alpha)$ and a
  nonzero measurable function $\varphi_2:A_0\to[0,1]$ such that $\varphi_2/\eta$ is bounded, such that, for all probability measure
  $\mu$ on $E$, all $f\in\mathcal{B}_{V}$, all $n\in\mathbb{Z}_+$ and all $j\in\{0,1,\ldots,t-1\}$,
  \begin{align}
    \left|\frac{\mu P_{nt+j}f}{\mu P_{nt+j}\11_E} -\frac{\sum_{i=0}^{t-1}(\restriction{\mu}{A_i}P_{t-i}\eta)\nu
        P_{i+j}f}{\sum_{i=0}^{t-1}(\restriction{\mu}{A_i} P_{t-i}\eta)\nu P_{i+j}\mathbbm{1}_E}\right|
    &\leq C\,\bar{\alpha}^n\frac{\sum_{i=0}^{t-1}\restriction{\mu}{A_i} P_{t-i} V}{\sum_{i=0}^{t-1}\restriction{\mu}{A_i}
      P_{t-i}\varphi_2}.
      \label{eq:main-1}
  \end{align}
  If in addition $n$ is large enough so that
  \begin{equation}
    \label{eq:hyp-main}
    C'_Q\theta_0^{-4t}\alpha^n\frac{\sum_{i=0}^{t-1}\restriction{\mu}{A_i} P_{t-i} V}{\sum_{i=0}^{t-1}\restriction{\mu}{A_i}
      P_{t-i}\eta}\leq \frac{1}{2},
  \end{equation}
  where the constants $C_Q$ and $\alpha$ are those from~\eqref{eq:lemme-note},
  there exists a constant $C'$ independent of $n$ and $\mu$ such that, for all probability measure $\mu$ on $E$, all
  $f\in\mathcal{B}_{V}$ and all $j\in\{0,1,\ldots,t-1\}$,
  \begin{align}
    \left|\frac{\mu P_{nt+j}f}{\mu P_{nt+j}\11_E} -\frac{\sum_{i=0}^{t-1}(\restriction{\mu}{A_i}P_{t-i}\eta)\nu
    P_{i+j}f}{\sum_{i=0}^{t-1}(\restriction{\mu}{A_i} P_{t-i}\eta)\nu P_{i+j}\mathbbm{1}_E}\right|
    &\leq C'\,\alpha^n\frac{\sum_{i=0}^{t-1}\restriction{\mu}{A_i} P_{t-i} V}{\sum_{i=0}^{t-1}\restriction{\mu}{A_i} P_{t-i}\eta}.
      \label{eq:main-2}
  \end{align}
\end{thm}

\begin{rem}
  \label{rem:main}
  As will appear in the proof below, the function $\varphi_2$ in the last result can be chosen positive in a large part of the
  support of $\eta$. More precisely, for any fixed $\varepsilon>0$, the function $\varphi_2$ can be chosen such that
  \begin{equation*}
    \inf\left\{\varphi_2(x),x\in E,\eta(x)\geq\varepsilon,V(x)\leq 1/\varepsilon\right\}>0.
  \end{equation*}
\end{rem}

\begin{proof}
  Let us first construct the function $\varphi_2$. Let $\theta_2\in(0,1)$ be such that
  \[
    \left(\frac{\theta_0}{\theta_2}\right)^t \alpha< 1
  \]
  and let $\varepsilon>0$ be small enough so that $\nu(K)\geq 1/2$, where
  \[
    K:=\left\{\eta\geq\varepsilon,\,V\leq \frac{1}{\varepsilon}\right\}.
  \]
  By~\eqref{eq:lemme-note}, there exists $n_0\in\mathbb{N}$ such that,
  \[
    \inf_{x\in K}\theta_2^{-n_0 t}\mathbb{P}_x(X_{n_0}\in K)\geq 1.
  \]
  Set for all $x\in K$
  \[
    \varphi_2(x):=\frac{\theta_2^{-t}-1}{\theta_2^{-n_0 t}-1}\sum_{k=0}^{n_0-1}\theta_2^{-k t}Q_k\mathbbm{1}_K(x).
  \]
  It is then easy to check (cf. e.g. Lemma 3.?? in~\cite{ChampagnatVillemonais2017b}) that, for all $x
  \in A_0$,
  \begin{equation}
    \label{eq:phi_2-Lyap}
    Q_1 \varphi_2(x)\geq\theta_2^t\varphi_2(x).
  \end{equation}
  In addition, since $\mathbbm{1}_{K}\leq \eta/\varepsilon$ and $Q_1\eta=\theta_0\eta$,
  \begin{equation}
    \label{eq:phi_2<eta}
    \varphi_2(x)\leq \frac{\theta_2^{-t}-1}{(\theta_2^{-n_0 t}-1)(1-\theta_0/\theta_2)}\,\eta(x),
  \end{equation}
  so we have proved the properties of the function $\varphi_2$ stated in Theorem~\ref{thm:QLD-periodic} and Remark~\ref{rem:main}.

  Now, assume that $n$ satisfies~\eqref{eq:hyp-main}. It then follows from~\eqref{eq:main-eta-periodic} that there exists a
  constant $C$ such that
  \begin{multline*}
    \left|\frac{\mu P_{nt+j}f}{\mu P_{nt+j}\11_E} -\frac{\sum_{i=0}^{t-1}(\restriction{\mu}{A_i}P_{t-i}\eta)\nu
        P_{i+j}f}{\sum_{i=0}^{t-1}(\restriction{\mu}{A_i} P_{t-i}\eta)\nu P_{i+j}\mathbbm{1}_E}\right| \\
    \begin{aligned}
      \leq & \frac{\theta_0^{-(n-1)t}\mu P_{nt+j}|f|}{\theta_0^{-(n-1)t}\mu
        P_{nt+j}\mathbbm{1}_E\left[\sum_{i=0}^{t-1}(\restriction{\mu}{A_i}P_{t-i}\eta)\nu P_{i+j}\mathbbm{1}_E\right]} \times \\ &
      \qquad\qquad \times\left|\theta_0^{-(n-1)t}\sum_{i=0}^{t-1}\restriction{\mu}{A_i}
        P_{t-i}P_{(n-1)t}P_{i+j}\mathbbm{1}_E-\sum_{i=0}^{t-1}(\restriction{\mu}{A_i} P_{t-i}\eta)\nu P_{i+j}\mathbbm{1}_E\right| \\
      & +\frac{1}{\sum_{i=0}^{t-1}(\restriction{\mu}{A_i}P_{t-i}\eta)\nu P_{i+j}\mathbbm{1}_E}\times \\ & \qquad\qquad
      \times\left|\theta_0^{-(n-1)t}\sum_{i=0}^{t-1}\restriction{\mu}{A_i}
        P_{t-i}P_{(n-1)t}P_{i+j}f-\sum_{i=0}^{t-1}(\restriction{\mu}{A_i} P_{t-i}\eta)\nu P_{i+j}f\right| \\ & \leq
      C\alpha^{n}\frac{\sum_{i=0}^{t-1} \restriction{\mu}{A_i}P_{t-i} V}{\sum_{i=0}^{t-1} (\restriction{\mu}{A_i}P_{t-i}
        \eta) \nu P_{i+j}\mathbbm{1}_E}\left(1+\frac{\theta_0^{-(n-1)t}\mu P_{nt+j}|f|}{\theta_0^{-(n-1)t}\mu
          P_{nt+j}\mathbbm{1}_E\left[\sum_{i=0}^{t-1}(\restriction{\mu}{A_i}P_{t-i}\eta)\nu P_{i+j}\mathbbm{1}_E\right]}\right).
    \end{aligned}
  \end{multline*}
  Since $\nu P_{i+j}\mathbbm{1}_E\geq \nu P_{i+j}P_{2t-i-j}\mathbbm{1}_E=\nu Q_2\mathbbm{1}_E=\theta_0^{2t}>0$, we deduce that
  \begin{multline*}
    \left|\frac{\mu P_{nt+j}f}{\mu P_{nt+j}\11_E} -\frac{\sum_{i=0}^{t-1}(\restriction{\mu}{A_i}P_{t-i}\eta)\nu
        P_{i+j}f}{\sum_{i=0}^{t-1}(\restriction{\mu}{A_i} P_{t-i}\eta)\nu P_{i+j}\mathbbm{1}_E}\right| \\
    \leq C\alpha^{n}\frac{\sum_{i=0}^{t-1} \restriction{\mu}{A_i}P_{t-i} V}{\sum_{i=0}^{t-1} \restriction{\mu}{A_i}P_{t-i}
      \eta}\left(1+\frac{\theta_0^{-(n-1)t}\mu P_{nt+j}|f|}{\theta_0^{-(n-1)t}\mu
        P_{nt+j}\mathbbm{1}_E\left[\sum_{i=0}^{t-1}(\restriction{\mu}{A_i}P_{t-i}\eta)\nu P_{i+j}\mathbbm{1}_E\right]}\right)
  \end{multline*}
  for some constant $C$. Now, we use~\eqref{eq:hyp-main} to deduce from~\eqref{eq:main-eta-periodic} that
  \begin{align*}
    \theta_0^{-(n-1)t}\mu P_{nt+j}\mathbbm{1}_E
    & \geq \sum_{i=0}^{t-1}(\restriction{\mu}{A_i}P_{t-i}\eta)\nu P_{i+j}\mathbb{1}_E-C'_Q
      \theta_0^{-(t+j)}\alpha^n\sum_{i=0}^{t-1}\restriction{\mu}{A_i}P_{t-i}V \\
    & \geq \theta_0^{2t}\sum_{i=0}^{t-1}\restriction{\mu}{A_i}P_{t-i}\eta-C'_Q
      \theta_0^{-2t}\alpha^n\sum_{i=0}^{t-1}\restriction{\mu}{A_i}P_{t-i}V \\
    & \geq\frac{\theta_0^{2t}}{2}\sum_{i=0}^{t-1}(\restriction{\mu}{A_i}P_{t-i}\eta)\nu P_{i+j}\mathbb{1}_E 
  \end{align*}
  and
  \begin{align*}
    \theta_0^{-(n-1)t}\mu P_{nt+j}|f|
    & \leq \sum_{i=0}^{t-1}(\restriction{\mu}{A_i}P_{t-i}\eta)\nu
      P_{i+j}|f|+C'_Q\theta_0^{-(t+j)}\alpha^n \sum_{i=0}^{t-1}\restriction{\mu}{A_i}P_{t-i} V
    \\ & \leq \left(\nu(V)+\frac{\theta_0^{3t}}{2}\right)\sum_{i=0}^{t-1}\restriction{\mu}{A_i}P_{t-i} \eta,
  \end{align*}
  where we used the definition of $\mathcal{B}_V$ in the last inequality. Therefore,
  \[
     \left|\frac{\mu P_{nt+j}f}{\mu P_{nt+j}\11_E} -\frac{\sum_{i=0}^{t-1}(\restriction{\mu}{A_i}P_{t-i}\eta)\nu
        P_{i+j}f}{\sum_{i=0}^{t-1}(\restriction{\mu}{A_i} P_{t-i}\eta)\nu P_{i+j}\mathbbm{1}_E}\right| 
    \leq C \alpha^{n}\frac{\sum_{i=0}^{t-1} \restriction{\mu}{A_i}P_{t-i} V}{\sum_{i=0}^{t-1} \restriction{\mu}{A_i}P_{t-i}
      \eta}
  \]
  for some constant $C$, so we have proved~\eqref{eq:main-2} and, thanks to~\eqref{eq:phi_2<eta},~\eqref{eq:main-1} for $n$
  satisfying~\eqref{eq:hyp-main}.

  Assume now that~\eqref{eq:hyp-main} is not satisfied. Then, using~\eqref{eq:main-eta-periodic} in a similar way as above,
  \begin{multline*}
    \left|\frac{\mu P_{nt+j}f}{\mu P_{nt+j}\11_E} -\frac{\sum_{i=0}^{t-1}(\restriction{\mu}{A_i}P_{t-i}\eta)\nu
        P_{i+j}f}{\sum_{i=0}^{t-1}(\restriction{\mu}{A_i} P_{t-i}\eta)\nu P_{i+j}\mathbbm{1}_E}\right|
    \leq\frac{\theta_0^{-(n-1)t}\mu P_{nt+j}f}{\theta_0^{-(n-1)t}\mu P_{nt+j}\11_E}+\frac{\nu(V)}{\theta_0^{2t}} \\
    \leq \frac{\nu(V) \sum_{i=0}^{t-1}\restriction{\mu}{A_i}P_{t-i}\eta+C'_Q\theta_0^{-(t+j)}\alpha^n
      \sum_{i=0}^{t-1}\restriction{\mu}{A_i}P_{t-i}V }{\theta_0^{-(n-1)t}\sum_{i=0}^{t-1}\restriction{\mu}{A_i} P_{t-i}Q_n
      \restriction{(P_{i+j}\11_{A_{i+j}})}{A_0}}+\frac{\nu(V)}{\theta_0^{2t}},
  \end{multline*}
  where we made the abuse of notation that $A_{i+j}=A_{i+j-t}$ if $t\leq i+j\leq 2t-1$. Now, $\11_{A_{i+j}}\geq P_{2t-i-j}\11_{A_0}$,
  so for all $x\in A_0$,
  \[
    P_{i+j}\11_{A_{i+j}}\geq P_{2t}\11_{A_0}\geq P_{2t}\varphi_2\geq \theta_2^{2t}\varphi_2.
  \]
  Therefore,In additon, using~\eqref{eq:phi_2-Lyap} and that~\eqref{eq:hyp-main} is not satisfied,
  \[
    \left|\frac{\mu P_{nt+j}f}{\mu P_{nt+j}\11_E} -\frac{\sum_{i=0}^{t-1}(\restriction{\mu}{A_i}P_{t-i}\eta)\nu
        P_{i+j}f}{\sum_{i=0}^{t-1}(\restriction{\mu}{A_i} P_{t-i}\eta)\nu P_{i+j}\mathbbm{1}_E}\right|
    \leq C\alpha^n\left(\frac{\theta_0}{\theta_2}\right)^{(n-1)t}\frac{\sum_{i=0}^{t-1}\restriction{\mu}{A_i} P_{t-i}
      V}{\sum_{i=0}^{t-1}\restriction{\mu}{A_i} P_{t-i} \varphi_2},
  \]
  so~\eqref{eq:main-1} is proved with $\bar{\alpha}=\alpha\theta_0^t/\theta_2^t<1$.  
\end{proof}

\begin{rem}
  \label{rem:QLD}
  Under the assumptions of Theorem~\ref{thm:eta-periodic}, the conditional distributions of $X$ converge to a quasi-stationary
  distribution if and only if the measure
  \[
  \frac{\sum_{i=0}^{t-1}(\restriction{\mu}{A_i}P_{t-i}\eta)\nu
    P_{i+j}}{\sum_{i=0}^{t-1}(\restriction{\mu}{A_i} P_{t-i}\eta)\nu P_{i+j}\mathbbm{1}_E}
  \]
  does not depend on $j$. Indeed, in this case, this measure is a quasi-limiting distribution, hence the unique quasi-stationary
  distribution given in Proposition~\ref{prop:one-to-one-QSD}. Comparing the measures for $j=0$ and $j=1$ entails that, for all
  $f\in\mathcal{B}_V$, the equality
  \begin{multline*}
    \left(\sum_{i=0}^{t-1}(\restriction{\mu}{A_i}P_{t-i}\eta)\nu
      P_{i} f\right)\left(\sum_{i=0}^{t-1}(\restriction{\mu}{A_i}P_{t-i}\eta)\nu
      P_{i+1} \11_E\right) \\ =\left(\sum_{i=0}^{t-1}(\restriction{\mu}{A_i}P_{t-i}\eta)\nu
      P_{i+1} f\right)\left(\sum_{i=0}^{t-1}(\restriction{\mu}{A_i}P_{t-i}\eta)\nu
      P_{i} \11_E\right)    
  \end{multline*}
  should hold true. Choosing $f=\11_{A_k}$, we obtain for all $k\in\{0,\ldots, t-1\}$,
  \[
  \frac{\restriction{\mu}{A_i}P_{t-k+1}\eta}{\restriction{\mu}{A_i}P_{t-k}\eta}=\frac{\sum_{i=0}^{t-1}(\restriction{\mu}{A_i}P_{t-i}\eta)\nu
    P_{i+1} \11_E}{\sum_{i=0}^{t-1}(\restriction{\mu}{A_i}P_{t-i}\eta)\nu
    P_{i} \11_E}.
  \]
  Since the right-hand side, say $\gamma$, does not depend on $k$, we deduce that
  \[
  \restriction{\mu}{A_i}P_{t-k}\eta= a\gamma^{-i}
  \]
  for some constant $a>0$. We deduce that
  \[
  \gamma=\frac{\sum_{i=0}^{t-1}(\restriction{\mu}{A_i}P_{t-i}\eta)\nu
    P_{i+1} \11_E}{\sum_{i=0}^{t-1}(\restriction{\mu}{A_i}P_{t-i}\eta)\nu
    P_{i} \11_E}=\gamma\,\frac{\sum_{i=0}^{t-1}\gamma^{-(i+1)}\nu
    P_{i+1} \11_E}{\sum_{i=0}^{t-1} \gamma^{-i}\nu
    P_{i} \11_E}=\gamma\left(1+\frac{1-\gamma^{-t}\,\nu P_t\11_E}{\sum_{i=0}^{t-1} \gamma^{-i}\nu
      P_{i} \11_E}\right).
  \]
  Since $\nu P_t=\theta_0^t \nu$, this equality is possible only if $\gamma=\theta_0$. Therefore, under the assumptions of
  Theorem~\ref{thm:QLD-periodic}, given $X_0\sim\mu$ such that $\sum_{i=0}^{t-1}\restriction{\mu}{A_i} P_{t-i} V<+\infty$ and
  $\sum_{i=0}^{t-1}\restriction{\mu}{A_i} P_{t-i}\eta>0$, the conditional distributions of $X$ converge in total variation to a
  quasi-stationary distribution if and only if $\theta_0^{-i}\restriction{\mu}{A_i}P_{t-i}\eta$ does not depend of
  $i\in\{0,\ldots,t-1\}$.
\end{rem}

\section{Spectral properties of periodic Markov chains under~\eqref{eq:lemme-note}}
\label{sec:spectral}


Let $\hat{P}_1 f(x)=\EE_x f(X_1)$ for all $x\in E\cup\{\d\}$ and all $f\in\hat{\mathcal{B}}_{\varphi_1}$, where
\[
\hat{\mathcal{B}}_{V}=\left\{f:E\cup\{\d\}\rightarrow\mathbb{C}\text{ s.t.\ }\exists a>0,\
\text{Re}(a\restriction{f}{E})\in\mathcal{B}_{V}\text{ and }\text{Im}(a\restriction{f}{E})\in\mathcal{B}_{V}\right\},
\]
where we denoted by $\text{Re}(z)$ and $\text{Im}(z)$ the real and imaginary parts of a complex number $z$, respectively.
Theorem~\ref{thm:eta-periodic} can be seen as a quasi-compactness property and it hence implies a spectral gap property, as shown in
the next result.

\begin{cor}
  \label{cor:periodic-spectral-gap}
  Under the assumptions of Theorem~\ref{thm:eta-periodic}, each eigenfunction $h\in\hat{\mathcal{B}}_{V}$ of $\hat{P}_1$ with
  eigenvalue $\theta\in\mathbb{C}$ satisfies the following properties:
  \begin{enumerate}
  \item if $h(\d)\neq 0$ and if $\PP_x(\tau_\d<\infty)=1$ for all $x\in E$, then $\theta=1$ and $h$ is constant;
  \item if $h(\d)=0$ and there exists $i\in\{0,\ldots,t-1\}$ such that $\nu P_i h=\nu P_i(\restriction{h}{A_i})\neq 0$, then
    $\theta=\theta_0$ and
    \begin{equation}
      \label{eq:thm-spectre}
      \restriction{h}{E}=\nu(h)\sum_{i=0}^{t-1}\theta_0^{-i}P_i\eta,
    \end{equation}
    where the eigenfunction $\eta$ of $Q$ has been extended by 0 out of $A_0$;
  \item if $h(\d)=0$, $\nu P_i h=0$ for  all $i\in\{0,\ldots,t-1\}$, then
    $|\theta|\leq \theta_0\alpha^{1/t}$.
  \end{enumerate}
\end{cor}

\begin{proof}
  Observe that, for all
  $f\in\hat{\mathcal{B}}_{\varphi_1}$ all $j\in\{0,\ldots,t-1\}$ and all $x\in E$,
  \begin{equation}
    \label{eq:calcul-trou-spectre}
    \hat{P}_j f(x)=P_j (\restriction{f}{E})(x)+f(\d)(1-P_j\11_E(x))
  \end{equation}
  and $\hat{P}_j f(\d)=f(\d)$.
  
  Assume first that $h(\d)\neq 0$. Then $\theta h(\d)=\hat{P}_1 h(\d)=h(\d)$, so $\theta=1$. Moreover, for all $x\in E$ and all
  $j\in\mathbb{N}$, 
  \[
    h(x)= P_j(\restriction{h}{E})(x)+h(\d)(1-P_j\11_E(x)).
  \]
  Now,~\eqref{eq:main-eta-periodic} and the fact that $\PP_x(\tau_\d<\infty)=1$ for all $x\in E$ implies that $\theta_0<1$ and, in
  particular, for all $x\in E$, $P_j(\restriction{h}{E})(x)$ and $P_j\11_E(x)$ both converge to 0 when $j\to+\infty$ and thus
  $h(x)=h(\d)$. Hence Point 1. is proved.
  
  Assume now that $h(\d)=0$ and that there exists $i\in\{0,\ldots,t-1\}$ such that $\nu P_i h\neq 0$. We can assume without loss of
  generality that $\nu P_i h>0$. It then follows from~\eqref{eq:calcul-trou-spectre} that
  $P_1(\restriction{h}{E})=\theta \restriction{h}{E}$. Let $k\in\{0,\ldots,t-1\}$ be fixed. Since
  $\theta_0^t\eta=P_t\eta=P_{k}P_{t-k}\eta$ is nonzero, there exists $x\in A_k$ such that $P_{t-k}\eta(x)>0$. For such an $x\in A_k$,
  it follow from~\eqref{eq:main-eta-periodic} that
  \[
    \lim_{n\to+\infty}\theta_0^{-(nt+i-k)}P_{nt+j}h(x)=
    \lim_{n\to+\infty}\left(\frac{\theta}{\theta_0}\right)^{-(nt+i-k)}h(x)\geq \theta_0^{-(t+i-k)}P_{t-k}\eta(x)\nu P_i h>0.
  \]
  Therefore, $\theta=\theta_0$. Since the convergence above holds for all $x
  \in A_k$,and, we have proved that, for all $x\in A_k$, $h(x)=\theta_0^{-(t+i-k)}P_{t-k}\eta(x)\nu P_i h$. In paticular,
  $\nu(h)=\theta_0^{-i}\nu(\eta)\nu P_i h=\theta_0^{-i}\nu P_i h>0$, so the previous computation could be done with $i=0$, thus
  entailing~\eqref{eq:thm-spectre}.

  Assume finally that $h(\d)=0$ and $\nu P_i h=0$ for all $i\in\{0,\ldots,t-1\}$. Let $k\in\{0,\ldots,t-1\}$ and $x\in A_k$ such that
  $h(x)\neq 0$. Then~\eqref{eq:main-eta-periodic} entails that, for all
  $j\in\{0,\ldots,t-1\}$,
  \[
    \left|\theta_0^{-(nt+j)}P_{nt+j}h(x)\right|=\left(\frac{|\theta|}{\theta_0}\right)^{nt+j}|h(x)|\leq C'_Q\alpha^n P_{t-k} V(x),
  \]
  hence $|\theta|\leq\theta_0 \alpha^{1/t}$ and Corollary~\ref{cor:periodic-spectral-gap} is proved.
\end{proof}

\section{Ergodic behavior of the $Q$-process of periodic Markov chains under~\eqref{eq:lemme-note}}
\label{sec:Q-proc}

We now study the $Q$-process of $X$, i.e.\ the law of $X$ conditioned to never be absorbed. We define $E'=\bigcup_{i=0}^{t-1}\{x\in
A_i,\, P_{t-i}\eta(x)>0\}$. Let us denote by $(\Omega,(\mathcal{F_n}_{n\geq 0},(X_n)_{n\geq 0},(\PP_x)_{x\in E})$ be a filtered
probability a space such that, under $\mathbb{P}_x$, the process $X$ has the same distribution as the Markov process with semi-group
$(P_n)_{n\geq 0}$ and initial condition $x$.

\begin{thm}
  \label{cor:Q-proc-periodic}
  Under the assumptions of Theorem~\ref{thm:eta-periodic}, we have the following properties:
  \begin{description}
  \item[\textmd{(i) Existence of the $Q$-process.}] There exists a family $(\QQ_x)_{x\in E'}$ of
    probability measures on $\Omega$ defined by
    $$
    \lim_{n\rightarrow+\infty}\PP_x(A\mid n<\tau_\partial)=\QQ_x(A)
    $$
    for all $x\in E'$, for all ${\cal F}_m$-measurable set $A$ and for all $m\geq 0$. The process $(\Omega,({\cal F}_n)_{n\geq
      0},(X_n)_{n\geq 0},(\QQ_x)_{x\in E'})$ is an $E'$-valued homogeneous Markov chain.
  \item[\textmd{(ii) Semigroup.}] The semigroup of the Markov process $X$ under $(\QQ_x)_{x\in E'}$ is given
    for all bounded measurable function $\varphi$ on $E'$, all $j\in\{0,\ldots,t-1\}$ and all $n\geq 0$ by
    \begin{align}
      \label{eq:semi-group-Q}
      \widetilde{P}_{nt+j}\varphi(x)=\sum_{i=0}^{t-1}\11_{A_i\cap E'}(x)\theta_0^{-(nt+j)}\frac{P_{nt+j}(\varphi \eta_{2t-i-j})(x)}{\eta_{t-i}(x)},
    \end{align}
    where, for all $k\in\mathbb{Z}_+$,
    \[
      \eta_k:=\theta_0^{-k}P_k\eta,
    \]
    where the function $\eta$ has been extended by 0 to $E\setminus A_0$.
  \item[\textmd{(iii) Exponential contraction in total variation.}] The Markov process $X$ under $(\QQ_x)_{x\in E'}$ admits as unique
    invariant probability measure
     \begin{align}
       \label{eq:inv-Q-proc}
       \frac{1-\theta_0}{1-\theta_0^t}\sum_{i=0}^{t-1}\nu P_i(\,\cdot\,\eta_{t-i}).
     \end{align}
     In addition, there exist constants $C<+\infty$ and $\alpha\in(0,1)$ such that, for all $j\in\{0,\ldots,t-1\}$, all $n\geq 0$,
     all probability measure $\mu'$ on $E'$ such that $\restriction{\mu'}{A_i}(\frac{P_{t-i}V}{\eta_{t-i}})<+\infty$ for all
     $i\in\{0,\ldots,t-1\}$ and all measurable $h:E'\rightarrow\RR$ such that, for all $k\in\{0,\ldots,t-1\}$ and all $x\in A_0$,
     $P_k(h\eta_{t-k})(x)\leq V(x)$ (or, equivalently, $\widetilde{P}_k h(x)\leq V(x)/\eta(x)$),
    \begin{align}
      \label{eq:Q-proc-11}
      \left|\mu'\widetilde{P}_{nt+j} h-\sum_{i=0}^{t-1} \theta_0^{t-i}\mu'(A_i) \nu P_{i+j}(h\eta_{2t-i-j})\right|\leq C
      \alpha^n\sum_{i=0}^{t-1}\restriction{\mu'}{A_i}\left(\frac{P_{t-i} V}{\eta_{t-i}}\right).
    \end{align}
    Furthermore, for all probability measure $\mu'$ on $E'$ and all $j\in\{0,\ldots,t-1\}$,
    \begin{align}
      \label{eq:Q-proc-12}
      \left\|\mu'\widetilde{P}_{nt+j}-\sum_{i=0}^{t-1} \theta_0^{t-i}\mu'(A_i) \nu P_{i+j}(\,\cdot\,
        \eta_{2t-i-j})\right\|_{TV}\xrightarrow[n\rightarrow\infty]{} 0,
    \end{align}
    where $\|\cdot\|_{TV}$ denotes the total variation norm.
  \end{description}
\end{thm}

\begin{proof}
  We introduce $\Gamma_m=\11_{m<\tau_\d}$ and define for all $x\in E'$ and $m,\geq 0$ the probability measure
  \begin{align*}
    Q^{\Gamma,x}_m=\frac{\Gamma_m}{\E_x\left(\Gamma_m\right)}\P_x,
  \end{align*}
  so that the $Q$-process exists if and only if $Q_m^{\Gamma,x}$ admits a proper limit when $m\rightarrow\infty$. Fix
  $j\in\{0,\ldots,t-1\}$ and $k\geq 0$. Let $k_0\in\{j-t+1,\ldots,j\}$ be such that $k=n_0 t+k_0$ for some integer $n_0$. For all
  $n\geq 0$ such that $nt+j\geq k$, we have by the Markov property
  \begin{align*}
    \frac{\E_x\left(\Gamma_{nt+j}\mid{\cal
          F}_k\right)}{\E_x\left(\Gamma_{nt+j}\right)} & =\11_{k<\tau_\d}\frac{\P_{X_k}\left(nt+j-k<\tau_\d\right)}{\P_x\left(nt+j<\tau_\d\right)}
    \\ & =\theta_0^{-k}\11_{k<\tau_\d}\frac{\theta_0^{-((n-n_0)t+j-k_0)}P_{(n-n_0)t+j-k_0}\11_E(X_k)}
    {\theta_0^{-(nt+j)}P_{nt+j}\11_E(x)}.
  \end{align*}
  Assume that $x\in A_i\cap E'$ for some $i\in\{0,\ldots,t-1\}$. By Theorem~\ref{thm:eta-periodic}, almost surely,
  \begin{align*}
    \lim_{n\rightarrow+\infty}\frac{\E_x\left(\Gamma_{nt+j}\mid{\cal
          F}_k\right)}{\E_x\left(\Gamma_{nt+j}\right)}
    & =\theta_0^{-k}\11_{k<\tau_\d}\frac{\theta_0^{-(2t+j-k_0)}P_{2t-(i+k_0)}\eta(X_k)\,\nu P_{i+j}\11_E}
    {\theta_0^{-(t+j)}P_{t-i}\eta(x)\,\nu P_{i+j}\11_E} \\ & =\theta_0^{-(n_0+1)t}\11_{k<\tau_\d}\frac{P_{2t-(i+k_0)}\eta(X_k)}
    {P_{t-i}\eta(x)}=:M_k.
  \end{align*}
  Since the limit is independent of $j\in\{0,\ldots,t-1\}$, we deduce that $\frac{\E_x\left(\Gamma_{m}\mid{\cal
        F}_k\right)}{\E_x\left(\Gamma_{m}\right)}\rightarrow M_k$ almost surely when $m\rightarrow+\infty$. Since in addition 
  \[
  \EE_x
  M_k=\theta_0^{-(t+k-k_0)}\frac{P_{2t-i+k-k_0}\eta(x)}{P_{t-i}\eta(x)}=1,
  \]
  we can apply the penalization's theorem of Roynette, Vallois and Yor \cite[Theorem~2.1]{RoynetteValloisEtAl2006}, which implies
  that $M$ is a martingale under $\P_x$ and that $Q_m^{\Gamma,x}(A)$ converges to $\E_x\left(M_k\11_{A}\right)$ for all $A\in{\cal
    F}_k$ when $m\rightarrow\infty$. This means that $\Q_x$ is well defined and
  \begin{align}
    \restriction{\frac{d\Q_x}{d\P_x}}{{\cal F}_k} & =M_k.
    \label{eq:Q-proc-sg}
  \end{align}
  Note that, in view of the expression of $M_k$ and by definition of $E'$, $(X_n,n\geq 0)$ is $E'$-valued $\mathbb{Q}_x$-almost
  surely for all $x\in E'$. The fact that $X$ is Markov under $(\QQ_x)_{x\in E''}$ can be easily deduced from the last formula (see
  e.g.~\cite[Section\,6.1]{ChampagnatVillemonais2016b}).

  Point~(ii) is a direct consequence of~\eqref{eq:Q-proc-sg} and of the definition of $M_{nt+j}$.

  We can now prove~\eqref{eq:Q-proc-11}: this a direct consequence of~\eqref{eq:main-eta-periodic} with
  \[
    \mu(dx)=\sum_{i=0}^{t-1}\frac{1}{\eta_{t-i}(x)}\restriction{\mu'}{A_i}(dx).
  \]
  In particular, for all $x\in E'$,
  \[
    \left\|\delta_x\widetilde{P}_{nt+j}-\sum_{i=0}^{t-1} \theta_0^{t-i}\11_{A_i}(x) \nu
      P_{i+j}(h\eta_{2t-i-j})\right\|_{TV}\xrightarrow[n\to+\infty]{} 0,
  \]
  and so~\eqref{eq:Q-proc-12} follows from the dominated convergence theorem.

  It only remains to check that~\eqref{eq:inv-Q-proc} is the only invariant distribution for $\widetilde{P}$. Note that, because
  of~\eqref{eq:Q-proc-12}, all invariant measure must be of the form
  \[
  \sum_{i=0}^{t-1} a_i \nu P_{i+j}(\,\cdot\,
  \eta_{2t-i-j})
  \]
  for some $a_i\geq 0$ such that the last measure does not depend on $j\in\{0,\ldots,t-1\}$. Since $\nu P_{k}$ has support in $A_k$,
  identifying the last measure for $j=0$ and $j=1$, we deduce that all the constants $a_i$ must be equal, so that there is a unique
  invariant measure given by
  \[
  \frac{\sum_{i=0}^{j-1}\nu P_i(\, \cdot\, \eta_{t-i})}{\sum_{i=0}^{j-1}\nu P_i\eta_{t-i}}=\frac{\sum_{i=0}^{j-1}\nu P_i(\, \cdot\,
    \eta_{t-i})}{\sum_{i=0}^{j-1}\theta_0^i}. 
  \]
  This ends the proof of Theorem~\ref{cor:Q-proc-periodic}.
\end{proof}

\section{Quasi-ergodic behavior of periodic Markov chains under~\eqref{eq:lemme-note}}
\label{sec:QED}

Related to the asymptotic behavior of the $Q$-process is the so-called quasi-ergodicity~\cite{ChampagnatVillemonais2016}, given in
the next result.

\begin{thm}
  \label{thm:QED-periodic}
  Under the assumptions of Theorem~\ref{thm:eta-periodic}, there exists a constant $C<+\infty$ such that, for all bounded measurable
  $f:E\rightarrow[-1,1]$, all probability measure $\mu$ on $E$ such that $\restriction{\mu}{A_i}(\eta_{t-i})>0$ for some
  $i\in\{0,\ldots,t-1\}$ and $\restriction{\mu}{A_i} P_{t-i} V<+\infty$ for all $i\in\{0,\ldots,t-1\}$, for all $N\geq 0$
  \begin{multline}
    \left|\EE_\mu\left[\left.\frac{1}{N+1}\sum_{m=0}^{N}f(X_m)
          \quad \right|\
      N<\tau_\d\right]-\frac{\theta_0^{-t}}{t}\sum_{\ell=0}^{t-1}\nu P_\ell(fP_{t-\ell}\eta)\right| \\ \leq\frac{C
    \sum_{i=0}^{t-1}\restriction{\mu}{A_i}P_{t-i} V} 
  {(N+1)\,\sum_{i=0}^{t-1}\restriction{\mu}{A_i}(\eta_{t-i}) 
  }.
    \label{eq:QED-periodic}
  \end{multline}
\end{thm}

\begin{rem}
  \label{rem:QED=proba}
  Notice that, for $f\equiv 1$, $\sum_{k=0}^{t-1}\nu P_k(fP_{t-k}\eta)=\theta_0^t \sum_{k=0}^{t-1}\nu(\eta)=t\theta_0^t$, so the
  measure $\nu_\text{QE}:=\frac{\theta_0^{-t}}{t}\sum_{k=0}^{t-1}\nu P_k(\cdot P_{t-k}\eta)$ is a probability measure, called the
  \emph{quasi-ergodic distribution}.
\end{rem}

In the last result, we do not recover the full the quasi-ergodic theorem of~\cite{BreyerRoberts1999} since they obtain convergence for all
$f\in L^1(\nu_\text{QE})$. However, it does not seem that condition~\eqref{eq:lemme-note} is sufficient to imply the conditions
of~\cite{BreyerRoberts1999} (see~\cite{BCOV-2022+}). However,~\eqref{eq:lemme-note} allows to improve the
convergence in theorem~\ref{thm:QED-periodic} into what could be called a convergence in conditional probability.

\begin{cor}
  \label{cor:QED-periodic}
  Under the assumptions of Theorem~\ref{thm:eta-periodic}, for all bounded measurable
  $f:E\rightarrow\mathbb{R}$, all probability measure $\mu$ on $E$ such that $\restriction{\mu}{A_i}(\eta_{t-i})>0$ for some
  $i\in\{0,\ldots,t-1\}$ and $\restriction{\mu}{A_i} P_{t-i} V<+\infty$ for all $i\in\{0,\ldots,t-1\}$,
  \begin{equation}
    \label{eq:1}
    \lim_{N\rightarrow 0}\mathbb{P}_\mu\left(\left.\left|\frac{1}{N+1}\sum_{m=0}^{N}f(X_m)-\frac{\theta_0^{-t}}{t}\sum_{\ell=0}^{t-1}\nu
        P_\ell(fP_{t-\ell}\eta)\right|>\varepsilon\ \right|\ N<\tau_\d\right)=0,\quad\forall \varepsilon>0.
  \end{equation}
\end{cor}

\begin{proof}[Proof of Theorem~\ref{thm:QED-periodic}]
  In all the proof, the constant $C$ denotes a constant that may change from line to line. First notice that
  \begin{align*}
    \sum_{j=0}^{t-1}\theta_0^{-(i+j)}\nu P_{i+j}(f
    \eta_{2t-i-j}) & =\theta_0^{-2t}\sum_{j=0}^{t-1}\nu P_{i+j}(f
    P_{2t-i-j}\eta) \\ & =\theta_0^{-t}\left(\sum_{j=0}^{t-1-i}\nu P_{i+j}(f
    P_{t-i-j}\eta)+\sum_{j=t-i}^{t-1}\nu P_{i+j-t}(f
    P_{2t-i-j}\eta)\right) \\ & =\theta_0^{-t}\sum_{k=0}^{t-1}\nu P_k(fP_{t-k}\eta).
  \end{align*}
  Hence, introducing $n\geq 0$ and $k\in\{0,\ldots,t-1\}$ such that $N=nt+k$ and denoting
  $\bar{\mu}(\bar{\eta})=\sum_{i=0}^{t-1}\restriction{\mu}{A_i}(\eta_{t-i})\,\theta_0^{-(k+i)}\nu P_{k+i}\11_E$, it follows from
  Theorem~\ref{thm:eta-periodic} that
  \begin{align}
    & \left|\EE_\mu\left[\left.\frac{1}{N+1}\sum_{m=0}^{N}f(X_m)
    \ \right|\
    nt+k<\tau_\d\right]-\frac{\theta_0^{-t}}{t}\sum_{\ell=0}^{t-1}\nu P_\ell(fP_{t-\ell}\eta)\right| \notag \\
    & = \left|\EE_\mu\left[\left.\frac{1}{nt+k+1}\sum_{m=0}^{nt+k}f(X_m)
          \ \right|\
      nt+k<\tau_\d\right]\right. \notag \\ & \qquad\qquad \qquad\qquad \left.-\frac{\sum_{i=0}^{t-1}\sum_{j=0}^{t-1}\restriction{\mu}{A_i}(\eta_{t-i})\,\theta_0^{-(k+i)}\nu
            P_{k+i}\11_E\,\theta_0^{-(i+j)}\nu P_{i+j}(f
            \eta_{2t-i-j})}{t\,\sum_{i=0}^{t-1}\restriction{\mu}{A_i}(\eta_{t-i})\,\theta_0^{-(k+i)}\nu P_{k+i}\11_E}\right| \notag \\ & \leq\frac{|\theta_0^{-N}\mu P_N\11_E-\bar{\mu}(\bar{\eta})|}{\theta_0^{-N}\mu
      P_N\11_E\,\bar{\mu}(\bar{\eta})}\,\frac{1}{N+1}\sum_{m=0}^N
    \theta_0^{-N} \mu P_N\11_E \notag \\
    & +\frac{1}{(N+1)\bar{\mu}(\bar{\eta})}\sum_{i=0}^{t-1}\sum_{j=0}^{t-1}\sum_{m=0}^{\lfloor (N-j)/t\rfloor}
    \left|\theta_0^{-(mt+j)}\restriction{\mu}{A_i}P_{mt+j}\left(f\theta_0^{-[(n-m)t+k-j]}P_{(n-m)t+k-j}\11_E\right) \right. \notag \\ &
    \qquad\qquad\qquad\qquad\qquad\qquad\qquad -\theta_0^{-(mt+j)}\restriction{\mu}{A_i}P_{mt+j}(f\eta_{2t-i-j})\,\theta_0^{-(k+i)}\nu
      P_{k+i}\11_E\Big| \notag \\
    & +\frac{1}{(N+1)\bar{\mu}(\bar{\eta})}\sum_{i=0}^{t-1}\sum_{j=0}^{t-1}\sum_{m=0}^{\lfloor (N-j)/t\rfloor} \theta_0^{-(k+i)}\nu
      P_{k+i}\11_E
    \left|\theta_0^{-(mt+j)}\restriction{\mu}{A_i}P_{mt+j}(f\eta_{2t-i-j}) \right. \notag \\ & 
      \qquad\qquad\qquad\qquad\qquad\qquad\qquad\qquad\qquad -\restriction{\mu}{A_i}(\eta_{t-i})\,\theta_0^{-(i+j)} \nu
      P_{i+j}(f\eta_{2t-i-j})\Big| \notag \\
    & +\frac{1}{\bar{\mu}(\bar{\eta})}\sum_{i=0}^{t-1}\sum_{j=0}^{t-1} \restriction{\mu}{A_i}(\eta_{t-i})\,\theta_0^{-(k+i)}\nu
            P_{k+i}\11_E\,\theta_0^{-(i+j)}\nu P_{i+j}(f
            \eta_{2t-i-j})
            \left|\frac{\lfloor (N-j)/t\rfloor+1}{N+1}-\frac{1}{t}\right|. \label{eq:calcul-horrible}
  \end{align}
  It can be checked using Theorem~\ref{thm:eta-periodic} that the first term of the right-hand-side is bounded by
  \[
  C\alpha^n\frac{\sum_{i=0}^{t-1}\restriction{\mu}{A_i}P_{t-i} V}
  {(N+1)\bar{\mu}(\bar{\eta})}
  \]
  and that each of the three other terms are bounded by
  \[
  \frac{C}{N+1}\,\frac{\sum_{i=0}^{t-1}\restriction{\mu}{A_i}P_{t-i} V}
  {(N+1)\bar{\mu}(\bar{\eta})}.
  \]
  The only non-immediate bound is for the second term of the right-hand-side of~\eqref{eq:calcul-horrible}, which is bounded by
  \begin{multline*}
    \frac{C}{(N+1)\bar{\mu}(\bar{\eta})}
    \sum_{i=0}^{t-1}\sum_{j=0}^{t-1}\sum_{m=0}^{\lfloor (N-j)/t\rfloor}
    C \alpha^n \theta_0^{-(mt+j)}\restriction{\mu}{A_i}P_{mt+j}(f P_{2t-i-j} V) \\ \leq
    \frac{C' \alpha^n}{(N+1) \bar{\mu}(\bar{\eta})} \sum_{i=0}^{t-1}\sum_{j=0}^{t-1}\sum_{m=0}^{\lfloor (N-j)/t\rfloor}
    \theta_0^{-(mt+j)}\restriction{\mu}{A_i}P_{(m+1)t+t-i} V.
  \end{multline*}
  by a first application of Theorem~\ref{thm:eta-periodic}. A second application of Theorem~\ref{thm:eta-periodic} then proves that
  \[
  \theta_0^{-(mt+j)}\restriction{\mu}{A_i}P_{(m+1)t+t-i} V\leq \theta_0^{-(t-i)}\restriction{\mu}{A_i}P_{t-i}\eta \nu( V)+C\alpha^n
  \restriction{\mu}{A_i}P_{t-i} V\leq C' \restriction{\mu}{A_i}P_{t-i} V.
  \]
  Noting that
  \[
    \bar{\mu}(\bar{\eta})\geq \sum_{i=0}^{t-1}\restriction{\mu}{A_i}(\eta_{t-i})\,
    \inf_{0\leq \ell\leq 2t}\theta_0^{-\ell}\nu P_\ell \11_E,
  \]
  we have proved Theorem~\ref{thm:QED-periodic}.
\end{proof}

\begin{proof}[Proof of Corollary~\ref{cor:QED-periodic}]
  In all the proof, the constant $C$ denotes a constant that may change from line to line. The result follows from Chebychev's bound and
  \begin{equation}
    \lim_{N\rightarrow+\infty}\EE_\mu\left[\left.\left(\frac{1}{N+1}\sum_{m=0}^{N}f(X_m)
          -\frac{\theta_0^{-t}}{t}\sum_{\ell=0}^{t-1}\nu P_\ell(fP_{t-\ell}\eta)\right)^2 \ \right|\ N<\tau_\d\right]=0,
    \label{eq:QED-periodic-squared}
  \end{equation}
  which follows from a similar computation as before. First set
  \[
    \bar{f}=f-\frac{\theta_0^{-t}}{t}\sum_{\ell=0}^{t-1}\nu
    P_\ell(fP_{t-\ell}\eta)=f-\nu_\text{QE}(f).
  \]
  We have
  \begin{multline}
    \EE_\mu\left[\left.\left(\frac{1}{N+1}\sum_{m=0}^{N}\bar{f}(X_m)\right)^2
        \ \right|\
      N<\tau_\d\right]=\frac{1}{(N+1)^2}\sum_{m=0}^{N}\EE_\mu\left[\bar{f}^2(X_m)\mid
      N<\tau_\d\right] \\
 +\frac{2}{(N+1)^2}\sum_{0\leq n< m\leq N}\EE_\mu\left[\bar{f}(X_n)\bar{f}(X_m)\mid
        N<\tau_\d\right] \label{eq:but}
  \end{multline}
  We shall use the bound, for all $0\leq n<m\leq N$ and all $i\in\{0,\ldots,t-1\}$,
  \begin{multline*}
    \EE_{\restriction{\mu}{A_i}}\left[\bar{f}(X_n)\bar{f}(X_m)\11_{N<\tau_\d}\right] \\
    \begin{aligned}
      & =\EE_{\restriction{\mu}{A_i}}\left[\bar{f}(X_n)\bar{f}(X_m)\11_{m<\tau_\d}\left(\PP_{X_m}(N-m<\tau_\d)-\theta_0^{N-m-k-j}\eta_{t-k}(X_m)\nu
            P_{k+j}\11_E \right)\right]
      \\ & +\theta_0^{N-m-k-j}\nu
            P_{k+j}\11_E\
            \EE_{\restriction{\mu}{A_i}}\left[\bar{f}(X_n)\11_{n<\tau_\d}\left(\E_{X_n}\left(\bar{f}(X_{m-n})\eta_{t-k}(X_{m-n})\right)\phantom{\theta_0^{k}}\right.\right.
                  \\ & \qquad\qquad\qquad\qquad\qquad\qquad\qquad\left.\left.-\theta_0^{m-n-k'-j'}\eta_{t-k'}(X_n)\nu
                  P_{k'+j'}(\bar{f}\eta_{t-k})\right)\right] \\
            & +\theta_0^{N-n-k-k'-j-j'}\nu P_{k+j}\11_E\,\nu P_{k'+j'}(\bar{f}\eta_{t-k})\
            \EE_{\restriction{\mu}{A_i}}\left(\bar{f}(X_n)\eta_{t-k'}(X_n)\right),
    \end{aligned}
  \end{multline*}
  where $j\in\{0,\ldots, t-1\}$ is such that $N-m-j\in t\mathbb{Z}$, $k\in\{0,\ldots,t-1\}$ is such that $i+m-k\in t\mathbb{Z}$,
  $j'=k-k'$ (such that $m-n-j'\in t\mathbb{Z}$) and $k'\in\{0,\ldots,t-1\}$ is such that $i+n-k'\in t\mathbb{Z}$.

  Let us first deal with the last term of the last equation: for all fixed $n\in\{0,\ldots, N-t\}$ and all $p\in\mathbb{Z}$ such that
  $pt-i>n$ et $(p+1)t-i-1\leq N$, using that $j'=k-k'$, that $k'$ only depends on $n$ and that
  $\theta_0^{-k-j}\nu P_{k+j}\11_E=\theta_0^{-N-i}\nu P_{N+i}\11_E$ since $N+i-k-j\in t\mathbb{Z}$,
  \begin{multline*}
    \sum_{m=pt-i}^{(p+1)t-i-1}\theta_0^{N-n-k-k'-j-j'}\nu P_{k+j}\11_E\,\nu P_{k'+j'}(\bar{f}\eta_{t-k})\
    \EE_{\restriction{\mu}{A_i}}\left(\bar{f}(X_n)\eta_{t-k'}(X_n)\right) \\
    =\theta_0^{-n-i}\,\nu P_{N+i}\11_E\,\restriction{\mu}{A_i} P_n(\bar{f}\eta_{t-k'})\sum_{m=pt-i}^{(p+1)t-i-1}\theta_0^{-k}\nu
    P_{k}(\bar{f}\eta_{t-k})=0,
  \end{multline*}
  since $\nu_\text{QE}(\bar{f})=0$.
  
  Combining the last two inequalities and using Theorem~\ref{thm:eta-periodic}, we deduce that, for all $i\in\{0,\ldots,t-1\}$,
  \begin{multline*}
    \frac{2}{(N+1)^2}\sum_{0\leq n< m\leq N}\EE_{\restriction{\mu}{A_i}}\left[\bar{f}(X_n)\bar{f}(X_m)\11_{N<\tau_\d}\right] \\
    \begin{aligned}
      & \leq\frac{C}{(N+1)^2}\sum_{0\leq n< m\leq N}
      \theta_0^{N-n-k-k'-j-j'}\nu P_{k+j}\11_E\,\nu P_{k'+j'}(\bar{f}\eta_{t-k})\
        \restriction{\mu}{A_i} P_n(\bar{f}\eta_{t-k'}) \\
      & +\frac{C\|f\|_\infty^2}{(N+1)^2}\sum_{0\leq n< m\leq N}
      \alpha^{N-m}\theta_0^{N-m}\restriction{\mu}{A_i} P_{m+t-k}
        V+\alpha^{m-n}\theta_0^{N-n-k-j}\restriction{\mu}{A_i} P_{n+t-k'} V \\
      & \leq \frac{C t N\,\|f\|_\infty^2\,\theta_0^N }{(N+1)^2}\,\sup_{0\leq\ell< 2t}(\nu
        P_\ell\11_E)\,\nu(V)\,\restriction{\mu}{A_i}(\eta_{t-i}) \\
      & +\frac{C\|f\|_\infty^2\, \theta_0^N}{(N+1)^2}\,\restriction{\mu}{A_i}P_{t-i} V\,\sum_{0\leq n<m\leq
        N}\left(\alpha^{N-m}+\alpha^{m-n}\right),
    \end{aligned}
  \end{multline*}
  where we used in the last inequality that $\theta_0^{-m}\restriction{\mu}{A_i} P_{m+t-k}V\leq C \restriction{\mu}{A_i} P_{t-i} V$ because
  $m+t-k=\ell t -i$ for some integer $\ell$ and $\theta_0^{\ell t}Q_\ell V(x)\leq C V(x)$ for all $x\in A_0$
  by~\eqref{eq:lemme-note}, and similarly for $\theta_0^{-n}\restriction{\mu}{A_i} P_{n+t-k'} V$. Using that
  \[
    \sum_{0\leq n<m\leq
      N}\left(\alpha^{N-m}+\alpha^{m-n}\right)\leq \frac{2N}{1-\alpha},
  \]
  we have proved that
  \[
   \frac{2}{(N+1)^2}\sum_{0\leq n< m\leq N}\EE_{\mu}\left[\bar{f}(X_n)\bar{f}(X_m)\11_{N<\tau_\d}\right] \leq\frac{C(\mu)\,\|f\|^2_\infty\,\theta_0^N}{N+1},
  \]
  where the constant $C(\mu)$ depends on $\mu$.

  Now, it follows from Theorem~\ref{thm:eta-periodic} that, for $N$ large enough, given $i\in\{0,\ldots,t-1\}$ such that
  $\restriction{\mu}{A_i}(\eta_{t-i})>0$,
  \[
    \PP_\mu(N<\tau_\d)\geq \PP_{\restriction{\mu}{A_i}}(N<\tau_\d)\geq\frac{1}{2}\restriction{\mu}{A_i}(\eta_{t-i})\theta_0^{N-i-\ell}\nu P_{i+\ell}
    \11_E
  \]
  where $\ell\in\{0,\ldots,t-1\}$ is such that $N-\ell\in t\mathbb{Z}$. Hence, there exists a constant $c(\mu)>0$ depending only on
  $\mu$ such that, for all $N$,
  $\PP_{\mu}(N<\tau_\d) \geq c(\mu)\theta_0^{N}$. Therefore,~\eqref{eq:QED-periodic-squared} follows from~\eqref{eq:but} and
  Corollary~\ref{cor:QED-periodic} is proved.
\end{proof}


\bibliographystyle{abbrv}
\bibliography{biblio-denis,biblio-math}

\end{document}